\documentclass{article}
\usepackage{cite}
\usepackage{amsmath,amsfonts,amssymb,amsthm,mathtools} 
\usepackage[affil-it]{authblk}

\newtheorem{theorem}{Theorem}[section]
\newtheorem{lemma}{Lemma}[section]

\def\authorfont{\footnotesize}
\def\keywords#1{\par
	\vspace*{8pt}
	{\authorfont{\leftskip18pt\rightskip\leftskip
			\noindent{\it Keywords}\/:\ #1\par}}\par}
\def\ccode#1{\par
	\vspace*{8pt}
	{\authorfont{\leftskip18pt\rightskip\leftskip
			\noindent #1\par}}\par}

\begin{document}

\date{\vspace{-5ex}}

\title{\MakeUppercase{On the endomorphisms and derivations of some\\ Leibniz algebras }}

\author{\footnotesize Leonid A. Kurdachenko}
\affil{Department of Geometry and Algebra, Faculty of Mechanics and Mathematics,\newline
Oles Honchar Dnipro National University, Gagarin ave., 72, Dnipro, 49010, Ukraine\newline
\texttt{lkurdachenko@i.ua}}

\author{Igor Ya. Subbotin}
\affil{Department of Mathematics and Natural Sciences, National University,\newline
5245 Pacific Concourse Drive, Los Angeles, California, 90045, USA\newline
\texttt{isubboti@nu.edu}}

\author{Viktoriia S. Yashchuk}
\affil{Department of Geometry and Algebra, Faculty of Mechanics and Mathematics,\newline 
Oles Honchar Dnipro National University, Gagarin ave., 72, Dnipro, 49010, Ukraine\newline
\texttt{Viktoriia.S.Yashchuk@gmail.com}}

\maketitle

\begin{abstract}
	We study the endomorphisms and derivations of infinite dimensional cyclic Leibniz algebra. 
\end{abstract}

\keywords{derivation, endomorphism, finitary matrix, ideal, Leibniz algebra, Lie algebra, monoid}

\ccode{2000 Mathematics Subject Classification: 17A32, 17A60, 17A99}

Let $L$ be an algebra over a field $F$ with the binary operations $+$ and $[\cdot,\cdot]$. Then $L$ is called \textit{a left Leibniz algebra}, if it satisfies the left Leibniz identity
\begin{equation*} 
[[a, [b, c]] = [a,[b, c]] - [b, [a, c]] \mbox{ for all } a, b, c \in L. 
\end{equation*} 

We will also use another form of this identity:
\begin{equation*} 
[[a, [b, c]] = [[a, b], c] + [b, [a, c]] \mbox{ for all } a, b, c \in L. 
\end{equation*} 
Leibniz algebras appeared first in the paper of A. Bloh \cite{BA1965}, but the term ``Leibniz algebra'' appears in the book of J.-L. Loday \cite{LJ1992} and his article \cite{LJ1993}. In \cite{LP1993} J. Loday and T. Pirashvili began the actual study of properties of Leibniz algebras. The theory of Leibniz algebras was developed very intensively within many different directions. Some of the results of this theory were presented in the book \cite{AOR2020}. Note that Lie algebras are a partial case of Leibniz algebras. Conversely, if $L$ is a Leibniz algebra in which $[a, a] = 0$ for every element $a \in L,$ then it is a Lie algebra. Thus, Lie algebras can be characterized as the anticommutative Leibniz algebras. The question about those properties of Leibniz algebras that the Lie algebra does not have and, accordingly, about those types of Leibniz algebras that have essential differences from Lie algebras naturally arises. A lot has already been done in this direction. We will not review the related results here, we simply link to the surveys \cite{KKPS2017, KSeSu2020} and the papers \cite{CKSu2017, CKSe2020, KOP2016, KOS2019, KSeSu2017, KSeSu2018, KSuSe2018, KSY2018, KSY2020, KSY2020A, YaV2019}. When studying Leibniz algebras,
the information about the endomorphisms and derivations of a Leibniz algebra is
quite useful.

Let $L$ be a Leibniz algebra. As usual, a linear transformation of $L$ is called \textit{an endomorphism}, if $f([a, b]) = [f(a), f(b)]$ for all $a, b \in L$. Clearly a product of two endomorphisms of $L$ is also an endomorphism, so that the set of all endomorphisms of $L$ is a semigroup by its multiplication. We note that the sum of two endomorphisms is not necessarily an endomorphism, so we cannot talk about an endomorphism ring.

Here we will use the term semigroup for a set, having an associative binary operation. For a semigroup, having an identity element, we will use the term \textit{monoid}. Clearly an identical permutation is an endomorphism of $L$, therefore the set $\mathbf{Lend}(L)$ of all endomorphisms of $L$ is a monoid by a multiplication.

As usual, a bijective endomorphism of $L$ is called \textit{an automorphism} of $L$.

Let $f$ be an automorphism of $L$, then it is possible to show that the mapping $f^{-1}$ is also an automorphism. Thus the set $\mathbf{Aut}(L)$ of all automorphisms of $L$ is a group by a multiplication.

Note that the endomorphisms of Leibniz algebras virtually were not studied. It was also quite unusual that the structure of cyclic Leibniz algebras was described relatively recently (see paper \cite{CKSu2017}). In the current paper we will consider endomorphisms of an infinite dimensional cyclic Leibniz algebra.

If $V$ is a vector space over a field $F$ having countable dimension, $\{v_n \mid n \in \mathbb N\}$ be a basis of $V$ and $f$ be a linear transformation of $V$, then 
\begin{equation*}
f(v_j) = \sum\limits_{n \in \mathbb N} \sigma_{nj}v_n,
\end{equation*} 
where among the coefficients $\sigma_{nj}$ only a finite subset of them are non-zero. As for finitedimensional vector spaces, we can talk about the matrix of a linear transformation $f$ on the basis $\{v_n \mid n \in \mathbb N\}$. This matrix will be infinite, but each of its columns has only a finite set of non-zero coefficients. Let us denote by $\mathbf{Mat}_{\mathbb N}(F)$ set of all matrices of this kind. As for finite-dimensional matrices, it is possible to define the product of such matrices, and as for finite-dimensional matrices, this product will be the matrix of the product of the corresponding linear transformations of a vector space $V$. Thus, the algebra $\mathbf{End}_{F}(V)$ of all linear transformations of $V$ is isomorphic to the algebra $\mathbf{Mat}_{\mathbb N}(F)$. This makes it possible to use the matrix apparatus for infinite-dimensional vector spaces. However, in contrast to the finite-dimensional case, the apparatus of infinitedimensional matrices is just beginning to be developed. The theory of infinitedimensional matrices over a field differs significantly from the finite-dimensional case. For example, an infinite-dimensional non-singular matrix may not have an inverse; more precisely, in the inverse matrix, its columns may contain an infinite set of non-zero coefficients.

The first our main result is the following.

\begin{theorem} \label{Theorem A}
	Let $L$ be a cyclic infinite dimensional Leibniz algebra over a field $F$. Then the monoid $\mathbf{Lend}(L)$ of all endomorphisms of $L$ is an union of an ideal $S$ with zero multiplication and a submonoid $\mathbf{Mon}(L)$ of all monomorphisms of $L$. Furthermore, $\mathbf{Mon}(L)$ is a product of an abelian submonoid $A$ and an abelian subgroup $D$, satisfying the following conditions:
	
	\begin{itemize}
		\item[$\mathrm{(i)}$] $A \cap D = \langle1\rangle$; 
		\item[$\mathrm{(ii)}$] $d^{-1}Ad = A$ for each element $d \in D$;
		\item[$\mathrm{(iii)}$] $D$ is isomorphic to a multiplicative group of a field $F$;
		\item[$\mathrm{(iv)}$] $A$ is isomorphic to a submonoid of a polynomial ring $F[X]$, consisting of those polynomials whose free term is $1$, in particular, $A$ is a free abelian monoid.
	\end{itemize}
\end{theorem}

\textbf{Corollary $\mathbf{A_{1}}$.} 
\textit{Let $L$ be a cyclic infinite dimensional Leibniz algebra over a field $F$. Then the group of all automorphisms of $L$ is isomorphic to a multiplicative group of a field $F$.}

Other linear transformations of a Leibniz algebra $L$ are the derivations. Denote by $\mathbf{End}_{F}(L)$ the set of all linear transformations of $L$, then $L$ is an associative algebra by the operations $+$ and $\circ$. As usual, $\mathbf{End}_{F}(L)$ is a Lie algebra by the operations $+$ and $[\cdot,\cdot]$, where $[f, g] = f \circ g - g \circ f$ for all $f, g \in \mathbf{End}_{F}(L)$.

A linear transformation $f$ of a Leibniz algebra $L$ is called a \textit{derivation}, if 
\begin{equation*}
f([a, b]) = [f(a), b] + [a, f(b)] \mbox{ for all } a, b \in L.
\end{equation*}

Let $\mathbf{Der}(L)$ be the subset of all derivations of $L$. It is possible to prove that $\mathbf{Der}(L)$ is a subalgebra of a Lie algebra $\mathbf{End}_{F}(L)$. $\mathbf{Der}(L)$ is called the \textit{algebra of derivations} of a Leibniz algebra $L$.

The derivations of Leibniz algebras also not very much studied, although their influence on the structure of the Leibniz algebra is more significant. This is indicated by the following result: if $A$ is an ideal of a Leibniz algebra, then the factor-algebra of $L$ by the annihilator of $A$ is isomorphic to some subalgebra of $\mathbf{Der}(A)$ \cite[Proposition 3.2]{KOP2016}.

Our second main result gives the description of an algebra of derivations of a cyclic infinite dimensional Leibniz algebra.

\begin{theorem} \label{Theorem B} 
	Let $L$ be a cyclic infinite dimensional Leibniz algebra over a field $F$. Then the Lie algebra $\mathbf{Der}(L)$ of all derivations of $L$ satisfies the following conditions:
	\begin{itemize}
		\item[$\mathrm{(i)}$] $L$ includes an abelian ideal \\
		$A = \{f \mid f \in \mathbf{Der}(L)\mbox{ and } f(x)\in[L, L] \mbox{ for all } x\in L\}$ \\
		and an abelian subalgebra $D$ such that $\mathbf{Der}(L) = A + D$ and $A \cap D = \langle0\rangle$;
		\item[$\mathrm{(ii)}$] $D$ is isomorphic to a field $F$;
		\item[$\mathrm{(iii)}$] if $\mathbf{char}(F) = 0$, then $[d, A] = A$ for each element $d \in D$.
	\end{itemize}
\end{theorem}

\section{Endomorphisms of a cyclic Leibniz algebra}

We recall some definitions.

Let $L$ be a Leibniz algebra. Define the \textit{lower central series}
\begin{equation*}
L=\gamma_{1}(L)\geqslant\gamma_{2}(L)\geqslant\ldots\gamma_{\alpha}(L)\geqslant\gamma_{\alpha+1}(L)\geqslant\ldots\gamma_{\delta}(L)
\end{equation*}
by the following rule: $\gamma_{1}(L)=L$, $\gamma_{2}(L)=[L,L]$, and recursively $\gamma_{\alpha+1}(L)=$ $=~[L,\gamma_{\alpha}(L)]$ for all ordinals $\alpha$ and $\gamma_{\lambda}(L)=\bigcap\limits_{\mu<\lambda}\gamma_{\mu}(L)$ for the limit ordinals $\lambda$. It is possible to shows that every term of this series is an ideal of $L$. The last term $\gamma_{\delta}(L)=\gamma_{\infty}(L)$ is called the \textit{lower hypocenter} of $L$. We have $\gamma_{\delta}(L)=[L,\gamma_{\delta}(L)]$.

If $\alpha=k$ is a positive integer, then $\gamma_{k}(L)=[L,[L,[L,\ldots]\ldots]]$ is the \textit{left normed commutator} of $k$ copies of $L$.

As usually, we say that a Leibniz algebra $L$ is called \textit{nilpotent}, if there exists a positive integer $k$ such that $\gamma_{k}(L) = \langle 0\rangle$. More precisely, $L$ is said to be \textit{nilpotent of nilpotency class \textbf{c}} if $\gamma_{\mathbf{c}+1}(L) = \langle 0\rangle$, but $\gamma_\mathbf{c}(L) \neq \langle 0\rangle$. 

The \textit{left} (respectively \textit{right}) \textit{center} $\zeta^{left}(L)$ (respectively $\zeta^{right}(L)$) of a Leibniz algebra $L$ is defined by the rule:
\begin{equation*}
\zeta^{left}(L)=\{x\in L \mid [x,y]=0 \mbox{ for each element } y\in L\}
\end{equation*}
(respectively,
\begin{equation*}
\zeta^{right}(L)=\{x\in L\mid [y,x]=0 \mbox{ for each element } y\in L\}).
\end{equation*}

It is not hard to prove that the left center of $L$ is an ideal, but it is not true for the right center. Moreover, $\mathbf{Leib}(L)\leqslant\zeta^{left}(L)$, so that $L/\zeta^{left}(L)$ is a Lie algebra. The right center is a subalgebra of $L$, and in general, the left and right centers are different; they even may have different dimensions (see \cite{KOP2016}).

The \textit{center} $\zeta(L)$ of $L$ is defined by the rule: 
\begin{equation*} 
\zeta(L) = \{ x \in L | [x, y] = 0 = [y, x] \mbox{ for each element } y \in L \}.
\end{equation*} 

The center is an ideal of $L$. 

Define the uppercentral series
\begin{equation*}
\langle0\rangle=\zeta_{0}(L)\leqslant\zeta_{1}(L)\leqslant\zeta_{2}(L)\leqslant\ldots\zeta_{\alpha}(L)\leqslant\zeta_{\alpha+1}(L)\leqslant\ldots\zeta_{\gamma}(L)=\zeta_{\infty}(L)
\end{equation*}
of a Leibniz algebra $L$ by the following rule: $\zeta_{1}(L)=\zeta(L)$ is the center of $L$, and recursively, $\zeta_{\alpha+1}(L)/\zeta_{\alpha}(L)=\zeta(L/\zeta_{\alpha}(L))$ for all ordinals $\alpha$, and \linebreak[5] $\zeta_{\lambda}(L)=\bigcup\limits_{\mu<\lambda}\zeta_{\mu}(L)$ for the limit ordinals $\lambda$. By definition, each term of this series is an ideal of $L$. The last term $\zeta_{\infty}(L)$ of this series is called the \textit{upper hypercenter} of $L$. If $L = \zeta_{\infty}(L)$ then $L$ is called \textit{a hypercentral} Leibniz algebra.

\begin{lemma} \label{Lemma 1.1} 
	Let $L$ be a Leibniz algebra over a field $F$ and $f$ be an automorphism of $L$. Then $f(\zeta^{left}(L)) =\zeta^{left}(L)$, $f(\zeta^{right}(L))=\zeta^{right}(L)$, $f(\zeta(L))=$ $=\zeta (L)$, $f([L, L]) = [L, L]$. 
\end{lemma}

\begin{proof}
	Let $x$ be an arbitrary element of $L$ and let $z \in \zeta^{left}(L)$. Since $f$ is an automorphism of $L$, there is an element $y \in L$ such that $x = f(y)$. Then we have
	\begin{equation*}
	[f(z), x] = [f(z), f(y)] = f([z, y]) = f(0) = 0.
	\end{equation*}
	It follows that $f(z) \in \zeta^{left}(L)$.
	
	There are the elements $u, v \in L$ such that $z = f(u)$, $x = f^{-1}(v)$. We have
	\begin{equation*}
	[u, x] = [f^{-1}(z), f^{-1}(v)] =f^{-1}([z, v]) = f^{-1}(0) = 0.
	\end{equation*}
	It follows that $u \in \zeta^{left}(L)$, so that $z \in f(\zeta^{left}(L))$, and, therefore, $\zeta^{left}(L) =$ $=f(\zeta^{left}(L))$.
	
	Using the similar arguments, we obtain the equality $\zeta^{right}(L) = f(\zeta^{right}(L))$ and $f(\zeta(L)) = \zeta(L)$.
	
	If $x, y$ are the elements of $L$, then $f([x, y]) = [f(x), f(y)] \in [L, L]$. It follows that $f([L, L]) \leqslant [L, L]$. Conversely, let $w \in [L, L]$, then 
	\begin{equation*}
		w = \alpha_{1}[u_{1}, v_{1}] + \ldots + \alpha_{t}[u_{t}, v_{t}] 
	\end{equation*}
	for some elements $u_{1}, v_{1}, \ldots , u_{t}, v_{t} \in L$, $\alpha_{1}, \ldots , \alpha_{t} \in F$. Since $f$ is an automorphism of $L$, there are the elements $a_{1}, b_{1}, \ldots, a_{t}, b_{t} \in L$ such that $u_{j} = f(a_{j})$, $v_{j} = f(b_{j})$, $1 \leqslant j \leqslant t$. We have
	\begin{equation*}
	\begin{gathered}
	w = \sum\limits_{1 \leqslant j \leqslant t}\alpha_{j}[u_{j}, v_{j}] = \sum\limits_{1 \leqslant j \leqslant t} \alpha_{j}[f(a_{j}), f(b_{j})] = \\ = \sum\limits_{1 \leqslant j \leqslant t} \alpha_{j}f([a_{j}, b_{j}]) = f\left(\sum\limits_{1 \leqslant j \leqslant t} \alpha_{j}[a_{j}, b_{j}]\right) \in f([L, L]).
	\end{gathered}
	\end{equation*}
	It follows that $[L, L] \leqslant f([L, L])$, and hence $[L, L] = f([L, L])$.
\end{proof}

\begin{lemma} \label{Lemma 1.2} 
	Let $L$ be a Leibniz algebra over a field $F$ and $f$ be an automorphism of $L$. Then $f(\zeta_{\alpha}(L)) = \zeta_{\alpha}(L)$, $f(\gamma_{\alpha}(L)) = \gamma_{\alpha}(L)$ for all ordinals $\alpha$, in particular, $f(\zeta_{\infty}(L)) = \zeta_{\infty}(L)$ and $f(\gamma_{\infty}(L))= \gamma_{\infty}(L)$.
\end{lemma}	
The proof is similar.

\begin{lemma} \label{Lemma 1.3} 
	Let $L$ be a Leibniz algebra over a field $F$ and $f$ be an endomorphism of $L$. Then $f(\gamma_{\alpha}(L)) \leqslant \gamma_{\alpha}(L)$ for all ordinals $\alpha$, in particular, $f(\gamma_{\infty}(L)) \leqslant\gamma_{\infty}(L)$.
\end{lemma}

\begin{proof}
	If $x, y$ are the elements of $L$, then $f([x, y]) = [f(x), f(y)] \in [L, L]$. It follows that $f([L, L]) \leqslant [L, L]$. Suppose that we have already proved that $f(\gamma_{\beta}(L)) \leqslant\gamma_{\beta}(L)$ for all ordinals $\beta<\alpha$. If $\alpha$ is a limit ordinal, then \linebreak[5] $\gamma_{\alpha}(L) = \bigcap\limits_{\beta<\alpha}\gamma_{\beta}(L)$. In this case,
	\begin{equation*}
	f(\gamma_{\alpha}(L))= f\left(\bigcap\limits_{\beta<\alpha}\gamma_{\beta}(L)\right) 	\leqslant \bigcap\limits_{\beta<\alpha} f(\gamma_{\beta}(L)) \leqslant \bigcap\limits_{\beta<\alpha}\gamma_{\beta}(L) = \gamma_{\alpha}(L).
	\end{equation*}
	
	Suppose now that $\alpha$ is not a limit ordinal, then $\alpha -1 = \delta$ exists. We have $\gamma_{\alpha}(L) =~[L, \gamma_{\delta}(L)]$. By induction hypothesis, $f(\gamma_{\delta}(L)) \leqslant \gamma_{\delta}(L)$. Let $w \in L$, $v \in \gamma_{\delta}(L)$, then $f([w, v]) = [f(w), f(v)] \in [L, \gamma_{\delta}(L)] = \gamma_{\alpha}(L)$. It follows that $f([L, \gamma_{\delta}(L)]) \leqslant \gamma_{\alpha}(L)$.
\end{proof}

\begin{lemma} \label{Lemma 1.4} 
	Let $L$ be a cyclic infinite dimensional Leibniz algebra over a field $F$, $L = \bigoplus\limits_{n \in \mathbb N} Fa_{n}$, where $[a_{1}, a_{n}] = a_{n+1}$ for all positive integer $n$, $[a_{m}, a_{k}] = 0$ for all $m > 1$, $k \in \mathbb N$. Then a linear mapping $f$ is an endomorphism of $L$ if and only if
	\begin{equation*}
	\begin{gathered}
	f(a_{1}) = \gamma_{1}a_{1} + \gamma_{2}a_{2} + \gamma_{3}a_{3} + \ldots + \gamma_{n - 1}a_{n - 1} + \gamma_{n}a_{n}, \\
	f(a_{2}) = \gamma_{1}^{2}a_{2} + \gamma_{1}\gamma_{2}a_{3} + \ldots + \gamma_{1}\gamma_{n-2}a_{n - 1} + \gamma_{1}\gamma_{n-1}a_{n} + \gamma_{1}\gamma_{n}a_{n+1}, \\
	f(a_{3}) = \gamma_{1}^{3}a_{3} + \gamma_{1}^{2}\gamma_{2}a_{4} + \ldots + \gamma_{1}^{2}\gamma_{n-2}a_{n} + \gamma_{1}^{2}\gamma_{n-1}a_{n + 1} + \gamma_{1}^{2}\gamma_{n}a_{n+2}, \\
	\mbox{and} \\
	f(a_{s}) = \gamma_{1}^{s}a_{s} + \gamma_{1}^{s - 1}\gamma_{2}a_{s + 1} + \gamma_{1}^{s - 1}\gamma_{3}a_{s + 2} + \ldots + \gamma_{1}^{s - 1} \gamma_{n}a_{n + s - 1}\\
	\mbox{ for all positive integer } s.
	\end{gathered}
	\end{equation*}
	
\end{lemma}

\begin{proof}
	Put $L_j= \bigoplus\limits_{n \geqslant j} Fa_{n}$. We have $\gamma_{1}(L) = L = L_{1}$, $\gamma_{2}(L) = L_{2}$, and, recursively, $\gamma_{n}(L) = L_{n}$ for all positive integer $n$. \textbf{Lemma \ref{Lemma 1.3}} shows that $f(L_{n}) \leqslant L_{n}$ for all positive integer $n$. We have
	$f(a_{1}) = \sum\limits_{n \in \mathbb N} \gamma_{n}a_{n}$ (here only finitely many coefficients $\gamma_{n}$ are non-zero).
	
	Then
	\begin{equation*}
	\begin{gathered}
	f(a_{2}) = f([a_{1}, a_{1}]) = \left[\sum\limits_{k \in \mathbb N} \gamma_{k}a_{k}, \sum\limits_{k \in \mathbb N} \gamma_{k}a_{k}\right] = \\
	= \left[\gamma_{1}a_{1}, \sum\limits_{k \in \mathbb N} \gamma_{k}a_{k}\right] =  \gamma_{1}\left(\sum\limits_{k \in \mathbb N} \gamma_{k}[a_{1}, a_{k}]\right) = \\
	= \gamma_{1}\left(\sum\limits_{k \in \mathbb N} \gamma_{k}a_{k+1}\right) =\sum\limits_{k \in \mathbb N} \gamma_{1}\gamma_{k}a_{k+1}. \\
	f(a_{3}) = f([a_{1}, a_{2}]) = \left[\sum\limits_{k \in \mathbb N} \gamma_{k}a_{k}, \sum\limits_{k \in \mathbb N} \gamma_{1}\gamma_{k}a_{k+1}\right] = \\
	=\left[\gamma_{1}a_{1},\sum\limits_{k \in \mathbb N} \gamma_{1}\gamma_{k}a_{k+1}\right] =\\ 
	= \gamma_{1}\left(\left[a_{1}, \sum\limits_{k \in \mathbb N} \gamma_{1}\gamma_{k}a_{k+1}\right]\right) = \sum\limits_{k \in \mathbb N} \gamma_{1}^{2}\gamma_{k}a_{k+2}.
	\end{gathered}
	\end{equation*}
	Using the similar arguments and ordinary induction we obtain that
	\begin{equation*}
	f(a_{s}) = \sum\limits_{k \in \mathbb N} \gamma_{1}^{s-1}\gamma_{k}a_{k+s-1}.
	\end{equation*}
	
	Conversely, let $x = \lambda_{1}a_{1} + \lambda_{2}a_{2} + \ldots + \lambda_{t}a_{t}$ and $y = \mu_{1}a_{1} + \mu_{2}a_{2} + \ldots + \mu_{m}a_{m}$ be arbitrary elements of $L$. Without loss of generality we may assume that $t = m$. Suppose that a linear mapping $f$ satisfies the above conditions. We have
	\begin{equation*}
	\begin{gathered} 
		[x, y] = \left[\lambda_{1}a_{1} + \lambda_{2}a_{2} + \ldots + \lambda_{t}a_{t}, \mu_{1}a_{1} + \mu_{2}a_{2} + \ldots + \mu_{t}a_{t}\right] = \\ = [\lambda_{1}a_{1}, \mu_{1}a_{1} + \mu_{2}a_{2} + \ldots + \mu_{t}a_{t}] = \\
		=\lambda_{1}\mu_{1}a_{2} + \lambda_{1}\mu_{2}a_{3} + \ldots + \lambda_{1}\mu_{t}a_{t+1};
	\end{gathered}
	\end{equation*}
	\begin{equation*}
	\begin{gathered}
		f([x, y]) = f(\lambda_{1}\mu_{1}a_{2} + \lambda_{1}\mu_{2}a_{3} + \ldots + \lambda_{1}\mu_{t}a_{t+1}) = \\ 
		= \lambda_{1}\mu_{1}f(a_{2}) + \lambda_{1}\mu_{2}f(a_{3}) + \ldots + \lambda_{1}\mu_{t}f(a_{t+1}) = \\ 
		= \lambda_{1}\mu_{1}\left(\sum\limits_{k \in \mathbb N} \gamma_{1}\gamma_{k}a_{k+1}\right) + \\
		+ \lambda_{1}\mu_{2}\left(\sum\limits_{k \in \mathbb N} \gamma_{1}^{2}\gamma_{k}a_{k+2}\right) + 
		\\ + \lambda_{1}\mu_{3}\left(\sum\limits_{k \in \mathbb N} \gamma_{1}^{3}\gamma_{k}a_{k+3}\right) + \ldots + \lambda_{1}\mu_{t}\left(\sum\limits_{k \in \mathbb N} \gamma_{1}^{t} \gamma_{k}a_{k+t}\right) = 
		\\ = \lambda_{1}\mu_{1}\gamma_{1}^{2}a_{2} + (\lambda_{1}\mu_{1}\gamma_{1}\gamma_{2} +
		\lambda_{1}\mu_{2}\gamma_{1}^{3})a_{3} + \\ 
		+ (\lambda_{1}\mu_{1}\gamma_{1}\gamma_{3} + \lambda_{1}\mu_{2}\gamma_{1}^{2}\gamma_{2} + \lambda_{1}\mu_{3}\gamma_{1}^{4})a_{4} + \ldots + \\ 
		+ (\lambda_{1}\mu_{1}\gamma_{1}\gamma_{s-1} + \lambda_{1}\mu_{2}\gamma_{1}^{2}\gamma_{s-2} + \ldots + \lambda_{1}\mu_{s-1}\gamma_{1}^{s})a_{s} + \ldots; 
		\end{gathered}
		\end{equation*}
		\begin{equation*}
		\begin{gathered}	 
		f(x) = f(\lambda_{1}a_{1} + \lambda_{2}a_{2} + \ldots + \lambda_{t}a_{t}) = 
		\\ = \lambda_{1}f(a_{1}) + \lambda_{2}f(a_{2}) + \ldots + \lambda_{t}f(a_{t}) = 
		\\ = \lambda_{1}\left(\sum\limits_{k \in \mathbb N} \gamma_{k}a_{k}\right) + \\ + \lambda_{2}\left(\sum\limits_{k \in \mathbb N} \gamma_{1}\gamma_{k}a_{k+1} \right) + \ldots +
		\\ + \lambda_{t}\left(\sum\limits_{k \in \mathbb N} \gamma_{1}^{t - 1} \gamma_{k}a_{k+t-1}\right); 
	\end{gathered}
	\end{equation*}
	\begin{equation*}
	\begin{gathered}	
		f(y) = f(\mu_{1}a_{1} + \mu_{2}a_{2} + \ldots + \mu_{t}a_{t}) = 
		\\ = \mu_{1}f(a_{1}) + \mu_{2}f(a_{2}) + \ldots + \mu_{t}f(a_{t}) = 
		\\ = \mu_{1}\left(\sum\limits_{k \in \mathbb N} \gamma_{k}a_{k}\right) +\\
		+ \mu_{2}\left(\sum\limits_{k \in \mathbb N} \gamma_{1}\gamma_{k}a_{k+1} \right) + \ldots +
		\\ + \mu_{t}\left(\sum\limits_{k \in \mathbb N} \gamma_{1}^{t-1}\gamma_{k}a_{k+t-1}\right);
	\end{gathered}
	\end{equation*}
	\begin{equation*}
	\begin{gathered}
		[f(x), f(y)] = \Bigg[\lambda_{1}\left(\sum\limits_{k \in \mathbb N} \gamma_{k}a_{k}\right) + \lambda_{2}\left(\sum\limits_{k \in \mathbb N} \gamma_{1}\gamma_{k}a_{k+1}\right) + \ldots + \\ +\lambda_{t}\left(\sum\limits_{k \in \mathbb N} \gamma_{1}^{t-1} \gamma_{k}a_{k+t-1}\right), 
		\mu_{1}\left(\sum\limits_{k \in \mathbb N} \gamma_{k}a_{k}\right)+ \\ +\mu_{2}\left(\sum\limits_{k \in \mathbb N} \gamma_{1}\gamma_{k}a_{k+1}\right) + \ldots + \mu_{t}\left(\sum\limits_{k \in \mathbb N} \gamma_{1}^{t-1} \gamma_{k}a_{k+t-1}\right)\Bigg] = \\
		= \Bigg[\lambda_{1}\gamma_{1}a_{1}, \mu_{1}\left(\sum\limits_{k \in \mathbb N} \gamma_{k}a_{k}\right) + \mu_{2}\left(\sum\limits_{k \in \mathbb N} \gamma_{1}\gamma_{k}a_{k+1} \right) + \ldots + \\
		+ \mu_{t}\left(\sum\limits_{k \in \mathbb N} \gamma_{1}^{t-1} \gamma_{k}a_{k+t-1}\right)\Bigg] = \\
		= \sum\limits_{k \in \mathbb N} \lambda_{1}\gamma_{1}\mu_{1}\gamma_{k}a_{k+1} + \sum\limits_{k \in \mathbb N} \lambda_{1}\gamma_{1}\mu_{2}\gamma_{1}\gamma_{k}a_{k+2} + \ldots + \\
		+\sum\limits_{k \in \mathbb N} \lambda_{1}\gamma_{1}\mu_{t}\gamma_{1}^{t-1} \gamma_{k}a_{k+t} = \\ 
		= \lambda_{1}\mu_{1}\gamma_{1}^{2}a_{2} + (\lambda_{1}\mu_{1}\gamma_{1}\gamma_{2} + \lambda_{1}\mu_{2}\gamma_{1}^{3})a_{3} + \\
		+ (\lambda_{1}\mu_{1}\gamma_{1}\gamma_{3} + \lambda_{1}\mu_{2}\gamma_{1}^{2}\gamma_{2} +
		+ \lambda_{1}\mu_{3}\gamma_{1}^{4})a_{4} + \ldots +	\\
		+ (\lambda_{1}\mu_{1}\gamma_{1}\gamma_{s-1} + \lambda_{1}\mu_{2}\gamma_{1}^{2}\gamma_{s-2} + \ldots + \lambda_{1}\mu_{s-1} \gamma_{1}^{s})a_{s} + \ldots 
	\end{gathered}
	\end{equation*}
	It follows that $f([x, y]) = [f(x), f(y)]$ for all elements $x, y \in L$.
\end{proof}

\textbf{Corollary 1.5.} \label{Corollary 1.5.}
\textit{Let $L$ be a cyclic infinite dimensional Leibniz algebra over a field $F$. Then the monoid of all endomorphisms of $L$ is an union of an ideal\linebreak[5] $S = \{f \mid f \in \mathbf{Lend}(L), f^{2} = 0\}$ and the submonoid $\mathbf{Mon}(L)$ consisting of all monomorphisms of $L$. Moreover, $S$ is an ideal with zero multiplication, i.e. $f\circ g = 0$ for every $f, g \in S$.}
\begin{proof}
	We have $L = \bigoplus\limits_{n \in \mathbb N} Fa_{n}$, where $[a_{1}, a_{n}] = a_{n + 1}$ for all positive integer $n$, $[a_{m}, a_{k}]= 0$ for all $m > 1$, $k \in \mathbb N$ \cite[Corollary 2.2 and 2.1]{CKSu2017}. Let $f$ be an endomorphism of $L$ and let $f(a_{1}) = \gamma_{1}a_{1} + \gamma_{2}a_{2} + \ldots + \gamma_{m}a_{m}$, where $\gamma_{j} \in F$, $1 \leqslant j \leqslant m$. Suppose that $\gamma_{1} = 0$, that is $f(a_{1}) \in [L, L] = \mathbf{Leib}(L)$. Then
	\begin{equation*}
	\begin{gathered}
	f(a_{2})= f([a_{1}, a_{1}] = [f(a_{1}), f(a_{1})] = 0, \\
	f(a_{3})= f([a_{1}, a_{3}]) = [f(a_{1}), f(a_{3})] = 0,
	\end{gathered}
	\end{equation*}
	and similarly we obtain that $f(a_{j}) = 0$ for all $j > 1$. It follows that $f(y) = 0$ for all elements $y \in [L, L]$. Then $f^{2}(a_{1}) = f(f(a_{1})) = 0$, $f^{2}(a_{j}) = f(f(a_{j})) = f(0) = 0$ for all $j > 1$. It follows that $f^{2}(x) = 0$ for all elements $x \in L$. This means that $f^{2}$ is a zero endomorphism.
	
	Conversely, let $f$ be an endomorphism of $L$ such that $f^{2} = 0$ and let again 
	\begin{equation*}
		f(a_{1}) = \gamma_{1}a_{1} + \gamma_{2}a_{2} + \ldots + \gamma_{m}a_{m}. 
	\end{equation*}
	We have
	\begin{equation*}
	\begin{gathered}
	f^{2}(a_{1}) = f(f(a_{1})) = f(\gamma_{1}a_{1} + \gamma_{2}a_{2} + \ldots + \gamma_{m}a_{m}) = 
	\\ = \gamma_{1}f(a_{1}) + \gamma_{2}f(a_{2}) + \ldots + \gamma_{m}f(a_{m}) = 
	\\ = \gamma_{1}^{2}a_{1} + \gamma_{1}\gamma_{2}a_{2} + \ldots + \gamma_{1}\gamma_{m}a_{m} + u, \mbox{ where } u \in [L, L].
	\end{gathered}
	\end{equation*}
	It follows that $f^{2}(a_{1})= \gamma_{1}^{2}a_{1} + v$, where $v \in [L, L]$. Since $Fa_{1}\cap [L, L] =\langle0\rangle$, $f^{2}= 0$ implies that $\gamma_{1}^{2}a_{1} = 0$ and $v = 0$. Thus $\gamma_{1}^{2} = 0$ and $\gamma_{1} = 0$. Hence 
	\begin{equation*}
	\begin{gathered}
	S = \{f \mid f \in \mathbf{Lend}(L), f^{2} = 0\} =
	\\ = \{f \mid f \in \mathbf{Lend}(L), f(x) \in [L, L] \mbox{ for each element } x \in L\}.
	\end{gathered}
	\end{equation*}
	
	Let $f \in S$ and $g$ be an arbitrary endomorphism of $L$. We have 
	\begin{equation*}
	(f\circ g)(x) = f(g(x)) \in [L, L],
	\end{equation*}
	and using \textbf{Lemma \ref{Lemma 1.3}} we obtain $(g \circ f)(x) = g(f(x)) \in [L, L]$. It follows that $S$ is an ideal of $\mathbf{Lend}(L)$. Moreover, let $f, g \in S$, then $(g \circ f)(x) = g(f(x)) = 0$, because $f(x) \in [L, L]$.
	
	Suppose now that $f\notin S$ and let again $f(a_{1}) = \gamma_{1}a_{1} + \gamma_{2}a_{2} + \ldots + \gamma_{m}a_{m}$. By above proved $\gamma_{1}\ne 0$. Let $x = \lambda_{1}a_{1} + \lambda_{2}a_{2} + \ldots + \lambda_{t}a_{t}$ be an arbitrary element of $L$, where $\lambda_{1}, \lambda_{2}, \ldots , \lambda_{t} \in F$. From the proof of \textbf{Lemma \ref{Lemma 1.4}} we obtain that 
	\begin{equation*}
	\begin{gathered}
	f(x) = \lambda_{1}\left(\sum\limits_{k \in \mathbb N} \gamma_{k}a_{k}\right) + \\
	+ \lambda_{2}\left(\sum\limits_{k \in \mathbb N} \gamma_{1}\gamma_{k}a_{k+1}\right) + \ldots +\\
	+ \lambda_{t}\left(\sum\limits_{k \in \mathbb N} \gamma_{1}^{t-1} \gamma_{k}a_{k+t-1}\right) = \\
	= \lambda_{1}\gamma_{1}a_{1} + (\lambda_{1}\gamma_{2} + \lambda_{2}\gamma_{1}^{2})a_{2} + \\
	+ (\lambda_{1}\gamma_{3} + \lambda_{2}\gamma_{1}\gamma_{2} +\lambda_{3}\gamma_{1}^{3})a_{3} + \ldots +
	\\+ \lambda_{t}\gamma_{1}^{t-1} \gamma_{m}a_{m+t-1} .
	\end{gathered}
	\end{equation*}
	An equality $f(x) = 0$ leads us to a system of linear equations, the matrix of which is obviously non-singular. Since such a system has only one zero solution, equality $f(x) = 0$ implies that $x = 0$. Thus every endomorphism $f$ such that $f \notin S$ is a monomorphism. Hence the monoid $\mathbf{Lend}(L)$ is an union of ideal $S$ and the subset $\mathbf{Mon}(L)$ of all monomorphisms of $L$.
	
	Clearly, a product of two monomorphisms is itself a monomorphism and an identical permutation is an automorphism of $L$. Therefore, set $\mathbf{Mon}(L)$ is a submonoid of $\mathbf{Lend}(L)$.
\end{proof}

\section{Proof of Theorem A}

We have $L = \bigoplus\limits_{n \in \mathbb N} Fa_{n}$, where $[a_{1}, a_{n}] = a_{n + 1}$ for all positive integer $n$, $[a_{m}, a_{k}] =$ $= 0$ for
all $m > 1$, $k \in \mathbb N$ \cite[Corollary 2.2 and 2.1]{CKSu2017}. By \textbf{Corollary 1.5.} the monoid $\mathbf{Lend}(L)$ is an union of an ideal $S$, consisting of all endomorphisms $h$ such that $h(x) \in [L, L]$, or, what is equivalent, $h^{2} = 0$, and a submonoid $\mathbf{Mon}(L)$, consisting of all monomorphism of $L$. A product of any two elements of $S$ is zero, so that its algebraic structure is trivial. Thus, we only must describe the submonoid $\mathbf{Mon}(L)$.

If $f$ is an arbitrary endomorphism of $L$, then \textbf{Lemma \ref{Lemma 1.4}} shows that $f$ has the following matrices in basis $\{a_{n} \mid n \in \mathbb N \}$:
\begin{equation*}
\begingroup
\setlength\arraycolsep{5pt}
\begin{pmatrix}
\gamma_{1} & 0 & 0 & 0 &\ldots & 0 & 0 & 0 & \ldots \\
\gamma_{2} & \gamma_{1}^{2} & 0 & 0 &\ldots & 0 & 0 & 0 & \ldots \\
\gamma_{3} & \gamma_{1}\gamma_{2} & \gamma_{1}^{3} & 0 &\ldots & 0 & 0 & 0 & \ldots \\
\gamma_{4} & \gamma_{1}\gamma_{3} & \gamma_{1}^{2}\gamma_{2} & \gamma_{1}^{4} &\ldots & 0 & 0 & 0 & \ldots \\
\gamma_{5} & \gamma_{1}\gamma_{4} & \gamma_{1}^{2}\gamma_{3} & \gamma_{1}^{3}\gamma_{2} &\ldots & 0 & 0 & 0 & \ldots \\
\gamma_{6} & \gamma_{1}\gamma_{5} & \gamma_{1}^{2}\gamma_{4} & \gamma_{1}^{3}\gamma_{3} &\ldots & 0 & 0 & 0 & \ldots \\
\gamma_{7} & \gamma_{1}\gamma_{6} & \gamma_{1}^{2}\gamma_{5} & \gamma_{1}^{3}\gamma_{4} &\ldots & 0 & 0 & 0 & \ldots \\
\vdots& \vdots& \vdots& \vdots& \ddots& \vdots& \vdots& \vdots& \vdots \\
\gamma_{n-2} & \gamma_{1}\gamma_{n-3} & \gamma_{1}^{2}\gamma_{n-4} & \gamma_{1}^{3}\gamma_{n-5} &\ldots & \gamma_{1}^{n-2} & 0 & 0 & \ldots \\
\gamma_{n-1} & \gamma_{1}\gamma_{n-2} & \gamma_{1}^{2}\gamma_{n-3} & \gamma_{1}^{3}\gamma_{n-4} &\ldots & \gamma_{1}^{n-3}\gamma_{2} & \gamma_{1}^{n-1} & 0 & \ldots \\
\gamma_{n} & \gamma_{1}\gamma_{n-1} & \gamma_{1}^{2}\gamma_{n-2} & \gamma_{1}^{3}\gamma_{n-3} &\ldots & \gamma_{1}^{n-3}\gamma_{3} & \gamma_{1}^{n-2}\gamma_{2} & \gamma_{1}^{n} & \ldots \\
0 & \gamma_{1}\gamma_{n} & \gamma_{1}^{2}\gamma_{n-1} & \gamma_{1}^{3}\gamma_{n-2} &\ldots & \gamma_{1}^{n-3}\gamma_{4} & \gamma_{1}^{n-2}\gamma_{3} & \gamma_{1}^{n-1}\gamma_{2} & \ldots \\
0 & 0 & \gamma_{1}^{2}\gamma_{n} & \gamma_{1}^{3}\gamma_{n-1} &\ldots & \gamma_{1}^{n-3}\gamma_{5} & \gamma_{1}^{n-2}\gamma_{4} & \gamma_{1}^{n-1}\gamma_{3} & \ldots \\
0 & 0 & 0 & \gamma_{1}^{3}\gamma_{n} &\ldots & \gamma_{1}^{n-3}\gamma_{6} & \gamma_{1}^{n-2}\gamma_{5} & \gamma_{1}^{n-1}\gamma_{4} & \ldots \\
0 & 0 & 0 & 0 &\ldots & \gamma_{1}^{n-3}\gamma_{7} & \gamma_{1}^{n-2}\gamma_{6} & \gamma_{1}^{n-1}\gamma_{5} & \ldots \\
\vdots& \vdots& \vdots& \vdots& \vdots& \vdots& \vdots& \vdots& \vdots \\
\end{pmatrix}
\endgroup
\end{equation*}

Moreover, if $f$ is a monomorphism, then \textbf{Corollary 1.5.} shows that $\gamma_{1}\ne 0$. And conversely, if a linear mapping $f$ has in basis $\{a_{n} \mid n \in \mathbb N \}$ this form, then \textbf{Lemma \ref{Lemma 1.4}} shows that $f$ is an endomorphism of $L$, and \textbf{Corollary 1.5.} shows that $\gamma_{1}\ne 0$ implies that $f$ is a monomorphism of $L$.

These matrices are finitary, so we can consider their multiplication. Since a product of two endomorphisms of $L$ is an endomorphism itself, we obtain that the matrices, having this form, form a monoid by its multiplication. Likewise the matrices, having this form with $\gamma_{1}\ne 0$, form a submonoid by multiplication. Denote this monoid
by $\mathbf{MC}(\infty)$. Thus we obtain an isomorphism \linebreak[5] $\mathbf{Mon}(L)\cong \mathbf{MC}(\infty)$.

Consider the set of matrices, having the following form
\begin{equation*}
\begingroup
\setlength\arraycolsep{8pt}
\begin{pmatrix}
1 & 0 & 0 & 0 &\ldots & 0 & 0 & 0 & \ldots \\
\gamma_{2} & 1 & 0 & 0 &\ldots & 0 & 0 & 0 & \ldots \\
\gamma_{3} & \gamma_{2} & 1 & 0 &\ldots & 0 & 0 & 0 & \ldots \\
\gamma_{4} & \gamma_{3} & \gamma_{2} & 1 &\ldots & 0 & 0 & 0 & \ldots \\
\gamma_{5} & \gamma_{4} & \gamma_{3} & \gamma_{2} &\ldots & 0 & 0 & 0 & \ldots \\
\gamma_{6} & \gamma_{5} & \gamma_{4} & \gamma_{3} &\ldots & 0 & 0 & 0 & \ldots \\
\gamma_{7} & \gamma_{6} & \gamma_{5} & \gamma_{4} &\ldots & 0 & 0 & 0 & \ldots \\
\vdots& \vdots& \vdots& \vdots& \ddots& \vdots& \vdots& \vdots& \vdots \\
\gamma_{n-2} & \gamma_{n-3} & \gamma_{n-4} & \gamma_{n-5} &\ldots & 1 & 0 & 0 & \ldots \\
\gamma_{n-1} & \gamma_{n-2} & \gamma_{n-3} & \gamma_{n-4} &\ldots & \gamma_{2} & 1 & 0 & \ldots \\
\gamma_{n} & \gamma_{n-1} & \gamma_{n-2} & \gamma_{n-3} &\ldots & \gamma_{3} & \gamma_{2} & 1 & \ldots \\
0 & \gamma_{n} & \gamma_{n-1} & \gamma_{n-2} &\ldots & \gamma_{4} & \gamma_{3} & \gamma_{2} & \ldots \\
0 & 0 & \gamma_{n} & \gamma_{n-1} &\ldots & \gamma_{5} & \gamma_{4} & \gamma_{3} & \ldots \\
0 & 0 & 0 & \gamma_{n} &\ldots & \gamma_{6} & \gamma_{5} & \gamma_{4} & \ldots \\
0 & 0 & 0 & 0 &\ldots & \gamma_{7} & \gamma_{6} & \gamma_{5} & \ldots \\
\vdots& \vdots& \vdots& \vdots& \vdots& \vdots& \vdots& \vdots& \vdots \\
\end{pmatrix}
\endgroup
\end{equation*}

Denote the set of all matrices, having this form by $\mathbf{UC}(\infty)$. 

Let $x$ be an arbitrary element of $L$, then $x = \lambda a_{1} + v$, where $v \in [L, L]$ and let $g$ be an endomorphism of $L$ such that $g(x) = \lambda a_{1} + u$, where $u \in [L, L]$. Clearly the matrix of $g$ in basis $\{a_{n} \mid n \in \mathbb N\}$ belong to $\mathbf{UC}(\infty)$, and conversely, every matrix from $\mathbf{UC}(\infty)$ define an endomorphism $g$ of $L$ such that $g(x) = \lambda a_{1} + u$, where $u \in [L, L]$. Denote by $A$ the subset of $S$, consisting of all endomorphisms with this property. If $h$ is another endomorphism of $L$ such that $h(x) = \lambda a_{1} + w$, where $w \in [L, L]$, then 
\begin{equation*}
(g\circ h)(x) = g(h(x)) = g(\lambda a_{1} + w) = \lambda a_{1} + z \mbox{ for some element } z \in [L, L].
\end{equation*} 
Taking into account the fact that the identity permutation of $L$ belongs to $A$, we obtain that $A$ is a submonoid of $\mathbf{Mon}(L)$. It follows that subset $\mathbf{UC}(\infty)$ is a submonoid of $\mathbf{MC}(\infty)$.

It is not hard to see that we can write every matrix from $\mathbf{UC}(\infty)$ in the form
\begin{equation*}
E +\gamma_{2}\sum\limits_{k \in \mathbb N}E_{k+1,k} + \gamma_{3}\sum\limits_{k \in \mathbb N}E_{k+2,k} + \ldots + \gamma_{n}\sum\limits_{k \in \mathbb N}E_{k+n-1,k}
\end{equation*}
Denote by $\mathbf{DmC}(\infty)$ the subset of $\mathbf{MC}(\infty)$, having the form $\sum\limits_{k \in \mathbb N}\gamma^{k}E_{k,k}$. It is not hard to see that $\mathbf{DmC}(\infty)$ is closed by multiplication. Denote by $D$ the subset of $\mathbf{Mon}(L)$, consisting of all endomorphisms $f$ such that $f(a_{1}) =~\gamma a_{1}$. Clearly, the matrix of every endomorphism $f \in D$ in basis $\{a_{n} \mid n \in \mathbb N\}$ belongs to $\mathbf{DC}(\infty)$, and conversely every matrix from $\mathbf{DmC}(\infty)$ defines an endomorphism $g \in D$.  Thus, $D\cong\mathbf{DC}(\infty)$. Clearly, the mapping
\begin{equation*}
\theta\colon \mathbf{DmC}(\infty) \longrightarrow \mathbf{U}(F), \mbox{ defined by the rule } \theta\left(\sum\limits_{k \in \mathbb N}\gamma^{k}E_{k,k}\right) = \gamma
\end{equation*}
is an isomorphism. It shows that $\mathbf{DmC}(\infty)$ is a subgroup of monoid $\mathbf{MC}(\infty)$ and $\mathbf{DmC}(\infty)$ is isomorphic to a multiplicative group of field $F$, in particular, it is abelian.

Consider now the product
\begin{equation*}
\begin{gathered}
\Bigg(E +\gamma_{1}^{-1}\bigg(\gamma_{2}\sum\limits_{k \in \mathbb N}E_{k+1,k} + \gamma_{3}\sum\limits_{k \in \mathbb N}E_{k+2,k} + \ldots + \\ +\gamma_{n}\sum\limits_{k \in \mathbb N}E_{k+n-1,k}\bigg)\Bigg) \cdot 
\left(\sum\limits_{t \in \mathbb N}\gamma_{1}^{t}E_{t,t}\right) = \\ =\sum\limits_{t \in \mathbb N}\gamma_{1}^{t}E_{t,t} + \left(\gamma_{1}^{-1}\gamma_{2}\sum\limits_{k \in \mathbb N}E_{k+1,k}\right)\left(\sum\limits_{t \in \mathbb N}\gamma_{1}^{t}E_{t,t}\right) + \\ +\left(\gamma_{1}^{-1}\gamma_{3}\sum\limits_{k \in \mathbb N}E_{k+2,k}\right) \left(\sum\limits_{t \in \mathbb N}\gamma_{1}^{t}E_{t,t}\right) + \ldots + \\ +\left(\gamma_{1}^{-1}\gamma_{n}\sum\limits_{k \in \mathbb N}E_{k+n-1,k}\right)\left(\sum\limits_{t \in \mathbb N}\gamma_{1}^{t}E_{t,t}\right) = \\
= \sum\limits_{t \in \mathbb N} \gamma_{1}^{t}E_{t,t} + \\
+(\gamma_{2}E_{2,1} E_{1,1}+ \gamma_{1}\gamma_{2}E_{3,2}E_{2,2} + \ldots + \gamma_{1}^{k - 1}\gamma_{2}E_{k+1,k}E_{k,k}+ \ldots ) + \\
+(\gamma_{3}E_{3,1}E_{1,1} + \gamma_{1}\gamma_{3}E_{4,2} E_{2,2} + \ldots + \gamma_{1}^{k - 1}\gamma_{3}E_{k+2,k}E_{k,k} + \ldots ) + \ldots + \\
+(\gamma_{n}E_{n,1} E_{1,1} + \gamma_{1}\gamma_{n}E_{n+1,2}E_{2,2} + \ldots + \gamma_{1}^{k - 1}\gamma_{n}E_{k+n-1,k}E_{k,k} + \ldots ) + \ldots = \\
= \sum\limits_{t \in \mathbb N}\gamma_{1}^{t}E_{t,t} + (\gamma_{2}E_{2,1} + \gamma_{1}\gamma_{2}E_{3,2} + \ldots + \gamma_{1}^{k - 1}\gamma_{2}E_{k+1,k} + \ldots ) + \\ 
+( \gamma_{3}E_{3,1} + \gamma_{1}\gamma_{3}E_{4,2} + \ldots + \gamma_{1}^{k - 1}\gamma_{3}E_{k+2,k} + \ldots ) + \ldots + \\
+(\gamma_{n}E_{n,1} + \gamma_{1}\gamma_{n}E_{n+1,2} + \ldots + \gamma_{1}^{k - 1}\gamma_{n}E_{k+n-1,k} + \ldots ).
\end{gathered}
\end{equation*}

These equalities show that monoid $\mathbf{MC}(\infty)$ is a product of submonoid $\mathbf{LC}(\infty)$ and abelian subgroup $\mathbf{DmC}(\infty)$; moreover, their intersection is trivial.

Let $\Gamma, \Lambda \in \mathbf{UC}(\infty)$ where
\begin{equation*}
\begin{gathered}
\Gamma = E +\sum\limits_{t \in \mathbb N} \left(\gamma_{t}\sum\limits_{k \in \mathbb N}E_{k+t,k}\right),
\\ \Lambda = E + \sum\limits_{t \in \mathbb N} \left(\lambda_{t}\sum\limits_{k \in \mathbb N}E_{k+t,k}\right)
\end{gathered}
\end{equation*}
(here only finitely many coefficients $\gamma_{t}, \lambda_{t}$ are non-zero). Let $n$ be the greatest number such that $\gamma_{n}\ne 0$, and, respectively, let $s$ be the greatest number  such~that $\lambda_{s} \ne 0$. Thus,
\begin{equation*}
\begin{gathered}
\Gamma = E +\gamma_{1}\sum\limits_{k \in \mathbb N}E_{k+1,k} + \gamma_{2}\sum\limits_{k \in \mathbb N}E_{k+2,k} + \ldots + \gamma_{n}\sum\limits_{k \in \mathbb N}E_{k+n,k} \\
\mbox{ and } \\
\Lambda =E + \lambda_{1}\sum\limits_{k \in \mathbb N}E_{k+1,k} + \lambda_{2}\sum\limits_{k \in \mathbb N}E_{k+2,k} + \ldots + \lambda_{s} \sum\limits_{k \in \mathbb N}E_{k+s,k}.
\end{gathered}
\end{equation*}
Put $\Gamma\Lambda = \sum\limits_{t \in \mathbb N} \left(\delta_{t}\sum\limits_{k \in \mathbb N}E_{k+t,k}\right)$. Since $\Gamma\Lambda \in \mathbf{UC}(\infty)$, then matrix $\Gamma\Lambda$ is completely defined by its first column. We have
\begin{equation*}
\begin{gathered} 
\delta_{j-1} = \gamma_{j- 1} + \gamma_{j- 2} \lambda_{1}+ \gamma_{j- 2} \lambda_{2}+ \ldots + \gamma_{1}\lambda_{j- 2} + \lambda_{j- 1} \\
\mbox { for each positive integer } j> 1.
\end{gathered}
\end{equation*}
In particular, $\delta_{n+s}= \gamma_{n}\lambda_{s}$ and $\delta_{j}= 0$ whenever $j>n+s$.

Taking all this into account, we arrive at the following isomorphism.
Define mapping $\Phi\colon \mathbf{UC}(\infty) \longrightarrow F[X]$ by the following rule:
\begin{equation*}
\begin{gathered}
\mbox{if } \Gamma \in \mathbf{UC}(\infty),\\
\Gamma = E + \gamma_{1}\sum\limits_{k \in \mathbb N}E_{k+1,k} + \gamma_{2}\sum\limits_{k \in \mathbb N}E_{k+2,k} + \ldots + \gamma_{n}\sum\limits_{k \in \mathbb N}E_{k+n,k}, \\
\mbox{ then put } \Phi(\Gamma) = 1 + \gamma_{1}X + \gamma_{2}X^{2} + \ldots + \gamma_{n}X^{n}.
\end{gathered}
\end{equation*}
By above proved, $\Phi(\Gamma\Lambda) = \Phi(\Gamma) \Phi(\Lambda)$ for every matrices $\Gamma, \Lambda \in \mathbf{UC}(\infty)$. Clearly the mapping $\Phi$ is injective. Further, $\mathbf{Im}(\Phi)$ is a submonoid of $F[X]$, which consists of those polynomials whose free term is $1$.

Let $M = \lVert\alpha_{kt}\rVert_{k,t \in \mathbb N}$ where only finitely many coefficients $\alpha_{kt}, k \in \mathbb N$ are non-zero for every index ${t \in \mathbb N}$, and $D = \lVert\delta_{kt}\rVert_{k, t \in \mathbb N}$, where $\delta_{kt} = 0$ whenever $k \ne t$. Put
\begin{equation*}
\begin{gathered}
\Lambda = D^{-1}M = \lVert\lambda_{kt}\rVert_{k, t \in \mathbb N}, \\
P = MD = \lVert\rho_{kt}\rVert_{k, t \in \mathbb N}, \\
\Sigma = D^{-1}MD = \lVert\sigma_{kt}\rVert_{k, t \in \mathbb N}.
\end{gathered}
\end{equation*}
We have
\begin{equation*}
\begin{gathered}
\lambda_{kt} = \sum\limits_{n \in \mathbb N}\delta^{-1}_{kn}\alpha_{nt} = \delta^{-1}_{k k}\alpha_{kt} , \\
\rho_{kt} = \sum\limits_{n \in \mathbb N}\alpha_{kn}\delta_{nt} = \alpha_{kt}\delta_{tt},\\
\sigma_{kt} = \sum\limits_{n \in \mathbb N}\lambda_{kn}\delta_{nt} = \lambda_{kt}\delta_{tt} = \delta^{-1}_{kk}\alpha_{kt}\delta_{tt}.
\end{gathered}
\end{equation*}
In particular, if
\begin{equation*}
\begin{gathered}
\Gamma = E +\gamma_{1}\sum\limits_{k \in \mathbb N}E_{k+1,k} + \gamma_{2}\sum\limits_{k \in \mathbb N}E_{k+2,k} + \ldots + \gamma_{n}\sum\limits_{k \in \mathbb N}E_{k+n,k},\\
D = \sum\limits_{k \in \mathbb N}\mu^{k}E_{k,k}, 
\end{gathered}
\end{equation*}
then
\begin{equation*}
\sigma_{21} = \mu^{-1}\gamma_{1}, \sigma_{31} = \mu^{-2}\gamma_{2}, \ldots ,\sigma_{n + 1,1} = \mu^{-n}\gamma_{n}.
\end{equation*}
Therefore, if $\Theta = E +\theta_{1}\sum\limits_{k \in \mathbb N}E_{k+1,k} + \theta_{2}\sum\limits_{k \in \mathbb N}E_{k+2,k} + \ldots + \theta_{n}\sum\limits_{k \in \mathbb N}E_{k+n,k}$, where
\begin{equation*}
\theta_{1} = \mu\gamma_{1}, \theta_{2} = \mu^{2}\gamma_{2}, \ldots ,\theta_{n} = \mu^{n}\gamma_{n}.
\end{equation*}
Then by above proved $D^{-1}\Theta D = \Gamma$. It follows that $D^{-1}\mathbf{UC}(\infty)D = \mathbf{UC}(\infty)$.

\section{Derivations of a cyclic Leibniz algebra}
\begin{lemma} \label{Lemma 2.1}
	Let $L$ be a Leibniz algebra over a field $F$ and $f$ be a derivation of $L$. Then
	$f(\zeta^{left}(L)) \leqslant\zeta^{left}(L)$, $f(\zeta^{right}(L))\leqslant\zeta^{right}(L)$ and $f(\zeta(L))\leqslant\zeta (L)$.
\end{lemma}

\begin{proof} 
	Let $x$ be an arbitrary element of $L$ and let $z \in \zeta^{left}(L)$. Then $[z, x] = 0$. Since a derivation is a linear mapping, $f([z, x]) = 0$. On the other hand,
	\begin{equation*}
	0 = f([z, x]) = [f(z), x] + [z, f(x)] = [f(z), x],
	\end{equation*}
	so that $f(z) \in \zeta^{left}(L)$.
	
	Let $z \in \zeta^{right}(L)$. Then $[x, z] = 0$. Now we have
	\begin{equation*}
	0 = f(0) = f([x, z]) = [f(x), z] + [x, f(z)] = [f(x), z],
	\end{equation*}
	so that $f(z) \in \zeta^{right}(L)$. Both above proved inclusions imply that
	\begin{equation*}
	f(\zeta(L))\leqslant\zeta(L).
	\end{equation*}
\end{proof}

\textbf{Corollary.} Let $L$ be a Leibniz algebra over a field $F$ and $f$ be a derivation of $L$. Then $f(\zeta_{\alpha}(L)) \leqslant\zeta_{\alpha}(L)$ for every ordinal $\alpha$.

\begin{lemma} \label{Lemma 2.3}
	Let $L$ be a Leibniz algebra over a field $F$ and $f$ be a derivation of $L$. Then $f(\gamma_{\alpha}(L)) \leqslant\gamma_{\alpha}(L)$ for all ordinals $\alpha$, in particular, $f(\gamma_{\infty}(L)) \leqslant\gamma_{\infty}(L)$.
\end{lemma}
\begin{proof}
	If $x, y$ are elements of $L$, then 
	\begin{equation*}
	f([x, y]) = [f(x), y] + [x, f(y)] \in [L, L]. 
	\end{equation*}
	It follows that $f([L, L] \leqslant [L, L]$. Suppose that we already proved that \linebreak[5] $f(\gamma_{\beta}(L)) \leqslant\gamma_{\beta}(L)$ for all ordinals $\beta<\alpha$. If $\alpha$ is a limit ordinal, then \linebreak[5] $\gamma_{\alpha}(L) = \bigcap\limits_{\beta<\alpha}\gamma_{\beta}(L)$. In this case,
	\begin{equation*}
	f(\gamma_{\alpha}(L))= f\left(\bigcap\limits_{\beta<\alpha}\gamma_{\beta}(L)\right) \leqslant\bigcap\limits_{\beta<\alpha}f(\gamma_{\beta}(L)) \leqslant\bigcap\limits_{\beta<\alpha}\gamma_{\beta}(L) = \gamma_{\alpha}(L).
	\end{equation*}
	Suppose now that $\alpha$ is a not limit ordinal, then $\alpha - 1 = \delta$ does exist. We have $\gamma_{\alpha}(L) = [L, \gamma_{\delta}(L)]$. By induction hypothesis, $f(\gamma_{\delta}(L)) \leqslant\gamma_{\delta}(L)$. Let $w \in L$, \linebreak[5] $v \in \gamma_{\delta}(L)$, then $f([w, v]) = [f(w), v] + [w, f(v)] \in [L, \gamma_{\delta}(L)] = \gamma_{\alpha}(L)$. It follows that $f([L, \gamma_{\delta}(L)]) \leqslant\gamma_{\alpha}(L)$.
\end{proof}
\begin{lemma} \label{Lemma 2.4} 
	Let $L$ be a cyclic infinite dimensional Leibniz algebra over a field $F$, $L =\bigoplus\limits_{n \in \mathbb N}Fa_{n}$, where $[a_{1}, a_{n}] = a_{n + 1}$ for all positive integer $n$, $[a_{m}, a_{k}] = 0$ for all $m > 1$, $k \in \mathbb N$.
	Then a linear mapping $f$ is a derivation of $L$ if and only if
	\begin{equation*}
	\begin{gathered}
	f(a_{1}) = \gamma_{1}a_{1} + \gamma_{2}a_{2} + \gamma_{3} a_{3} + \ldots + \gamma_{n-1} a_{n - 1}+ \gamma_{n}a_{n}, \\
	f(a_{2}) = 2\gamma_{1}a_{2} + \gamma_{2}a_{3} + \ldots + \gamma_{n-1} a_{n} + \gamma_{n}a_{n+1}, \\
	f(a_{3}) = 3\gamma_{1}a_{3} + \gamma_{2} a_{4} + \ldots + \gamma_{n-2}a_{n}+ \gamma_{n-1}a_{n + 1} + \gamma_{n}a_{n+2} , \\
	\mbox{ and } \\
	f(a_{s}) = s\gamma_{1}a_{s} + \gamma_{2}a_{s+1} + \gamma_{3}a_{s+2} + \ldots + \gamma_{n}a_{n+s-1} \\
	\mbox{ for all positive integers } s.
	\end{gathered}
	\end{equation*}
\end{lemma}

\begin{proof}
	Put $L_{j}= \bigoplus\limits_{n\geqslant j}Fa_{n}$. We have $\gamma_{1}(L) = L = L_{1}$, $\gamma_{2}(L) = L_{2}$, and, recursively, $\gamma_{n}(L) = L_{n}$ for all positive integer $n$. \textbf{Lemma \ref{Lemma 2.3}} shows that $f(L_{n}) \leqslant L_{n}$ for all positive integer $n$.
	We have
	\begin{equation*}
	\begin{gathered}
	f(a_{1}) = \sum\limits_{n \in \mathbb N}\gamma_{n}a_{n} \\
	\mbox{ (here only finitely many coefficients $\gamma_{n}$ are non-zero) }.
	\end{gathered}
	\end{equation*}
	Then
	\begin{equation*}
	\begin{gathered}
	f(a_{2}) = f([a_{1}, a_{1}]) = [f(a_{1}), a_{1}] + [a_{1}, f(a_{1})] = 
	\\ = \Bigg[ \bigg[\sum\limits_{k \in \mathbb N}\gamma_{k}a_{k}, a_{1}\bigg] + \bigg[a_{1}, \sum\limits_{k \in \mathbb N} \gamma_{k}a_{k}\bigg] \Bigg] = 
	\\ = [\gamma_{1}a_{1}, a_{1}] + \sum\limits_{k \in \mathbb N} \gamma_{k}[a_{1}, a_{k}] =\gamma_{1}a_{2} + \sum\limits_{k \in \mathbb N} \gamma_{k}a_{k+1} = \\
	= 2\gamma_{1}a_{2} + \sum\limits_{k \in \mathbb N} \gamma_{k+1}a_{k+2}, \\
	f(a_{3}) = f([a_{1}, a_{2}] = [f(a_{1}), a_{2}] + [a_{1}, f(a_{2})] = \\ = \left[\sum\limits_{k \in \mathbb N}\gamma_{k}a_{k}, a_{2}\right] + \left[a_{1}, 2\gamma_{1}a_{2} + \sum\limits_{k \in \mathbb N} \gamma_{k+1}a_{k+2}\right] = \\
	= [\gamma_{1}a_{1}, a_{2}] + 2\gamma_{1}[a_{1}, a_{2}] + \sum\limits_{k \in \mathbb N} \gamma_{k+1}[a_{1}, a_{k+2}] = \\ = 3\gamma_{1}a_{3} + \sum\limits_{k \in \mathbb N} \gamma_{k+1}a_{k+3}.
	\end{gathered}
	\end{equation*}
	Using the similar arguments and ordinary induction, we obtain that
	\begin{equation*}
	f(a_{s}) = s\gamma_{1}a_{s} + \sum\limits_{k \in \mathbb N} \gamma_{k + 1}a_{k+s}.
	\end{equation*}
	
	Conversely, let $x = \lambda_{1}a_{1} + \lambda_{2}a_{2} + \ldots + \lambda_{t}a_{t}$ and $y = \mu_{1}a_{1} + \mu_{2}a_{2} + \ldots + \mu_{m}a_{m}$ be arbitrary elements of $L$. Without loss of generality we may assume that $t = m$. Suppose that linear mapping $f$ satisfies the above conditions. We have 
	
	\begin{equation*}
	\begin{gathered}
	[x,y] = [\lambda_{1}a_{1} + \lambda_{2}a_{2} + \ldots + \lambda_{t}a_{t}, \mu_{1}a_{1} + \mu_{2}a_{2} + \ldots + \mu_{t}a_{t}] = \\
	= [\lambda_{1}a_{1}, \mu_{1}a_{1} + \mu_{2}a_{2} + \ldots + \mu_{t}a_{t}] = \lambda_{1}\mu_{1}a_{2} + \lambda_{1}\mu_{2}a_{3} + \ldots + \lambda_{1}\mu_{t}a_{t+1}; 
	\end{gathered}
	\end{equation*}
	
	\begin{equation*}
	\begin{gathered}
	f([x, y]) = f(\lambda_{1}\mu_{1}a_{2} + \lambda_{1}\mu_{2}a_{3} + \ldots + \lambda_{1}\mu_{t}a_{t+1}) = \\
	= \lambda_{1}\mu_{1}f(a_{2}) + \lambda_{1}\mu_{2}f(a_{3}) + \ldots + \lambda_{1}\mu_{t}f(a_{t+1}) = \\
	=\lambda_{1}\mu_{1}\left(2\gamma_{1}a_{2} + \sum\limits_{k \in \mathbb N} \gamma_{k+1}a_{k+2}\right) +\\
	+ \lambda_{1}\mu_{2}\left(3\gamma_{1}a_{3} + \sum\limits_{k \in \mathbb N} \gamma_{k+1}a_{k+3}\right) + \\
	+ \lambda_{1}\mu_{3}\left(4\gamma_{1}a_{4} + \sum\limits_{k \in \mathbb N} \gamma_{k+1}a_{k+4}\right) + \ldots + \\
	+\lambda_{1}\mu_{t}\left((t+1)\gamma_{1}a_{t+1} + \sum\limits_{k \in \mathbb N} \gamma_{k+1}a_{k+t+1}\right) = \\
	= 2\lambda_{1}\mu_{1}\gamma_{1}a_{2} + (\lambda_{1}\mu_{1}\gamma_{2} + 3\lambda_{1}\mu_{2}\gamma_{1})a_{3} + \\
	+(\lambda_{1}\mu_{1}\gamma_{3} + \lambda_{1}\mu_{2}\gamma_{2} + 4\lambda_{1}\mu_{3}\gamma_{1})a_{4} + \ldots + \\
	+(\lambda_{1}\mu_{1}\gamma_{t} + \lambda_{1}\mu_{2}\gamma_{t-1} + \ldots + (t+1)\lambda_{1}\mu_{t}\gamma_{1})a_{t+1} + \\
	+(\lambda_{1}\mu_{1}\gamma_{t+1} + \lambda_{1}\mu_{2}\gamma_{t} + \lambda_{1}\mu_{3}\gamma_{t-1} + \ldots + \lambda_{1}\mu_{t}\gamma_{2} )a_{t+2};
	\end{gathered}
	\end{equation*}	
	
	\begin{equation*}
	\begin{gathered}
	f(x) = f(\lambda_{1}a_{1} + \lambda_{2}a_{2} + \ldots + \lambda_{t}a_{t}) = \\ =\lambda_{1}f(a_{1}) + \lambda_{2}f(a_{2}) + \ldots + \lambda_{t}f(a_{t}) = \\
	=\lambda_{1}\left(\sum\limits_{k \in \mathbb N} \gamma_{k}a_{k}\right) + \\
	+ \lambda_{2}\left(2\gamma_{1}a_{2} + \sum\limits_{k \in \mathbb N} \gamma_{k+1}a_{k+2}\right) + \ldots + \\
	+ \lambda_{t}\left(t\gamma_{1}a_{t} + \sum\limits_{k \in \mathbb N} \gamma_{k+1}a_{k+t}\right);
	\end{gathered}
	\end{equation*}
	
	\begin{equation*}
	\begin{gathered}
	f(y) = f(\mu_{1}a_{1} + \mu_{2}a_{2} + \ldots + \mu_{t}a_{t}) = \\
	=\mu_{1}f(a_{1}) + \mu_{2}f(a_{2}) + \ldots + \mu_{t}f(a_{t}) = \\
	=\mu_{1}\left(\sum\limits_{k \in \mathbb N} \gamma_{k}a_{k}\right) + \\
	+\mu_{2}\left(2\gamma_{1}a_{2} + \sum\limits_{k \in \mathbb N} \gamma_{k+1}a_{k+2}\right) + \ldots + \\
	+\mu_{t}\left(t\gamma_{1}a_{t} + \sum\limits_{k \in \mathbb N} \gamma_{k+1}a_{k+t}\right);
	\end{gathered}
	\end{equation*}
	\begin{equation*}
	\begin{gathered}
	[f(x), y] = \Bigg[\lambda_{1}\left(\sum\limits_{k \in \mathbb N} \gamma_{k}a_{k}\right) + \lambda_{2}\left(2\gamma_{1}a_{2} + \sum\limits_{k \in \mathbb N} \gamma_{k+1}a_{k+2}\right) + \ldots + \\
	+\lambda_{t}\left(t\gamma_{1}a_{t} + \sum\limits_{k \in \mathbb N} \gamma_{k+1}a_{k+t}\right),\mu_{1}a_{1} + \mu_{2}a_{2} + \ldots + \mu_{t}a_{t}\Bigg] = \\
	= [\lambda_{1}\gamma_{1}a_{1},\mu_{1}a_{1} + \mu_{2}a_{2} + \mu_{3}a_{3} +\ldots + \mu_{t}a_{t}] = \\
	=\lambda_{1}\gamma_{1}\mu_{1}a_{2} + \lambda_{1}\gamma_{1}\mu_{2}a_{3} + \lambda_{1}\gamma_{1}\mu_{3}a_{4} +\ldots + \lambda_{1}\gamma_{1}\mu_{t}a_{t+1};
	\end{gathered}
	\end{equation*}
	
	\begin{equation*}
	\begin{gathered}[]
	[x, f(y)] = \Bigg[\lambda_{1}a_{1} + \lambda_{2}a_{2} + \ldots + \lambda_{t}a_{t}, \mu_{1}\left(\sum\limits_{k \in \mathbb N} \gamma_{k}a_{k}\right) + \\
	+ \mu_{2}\left(2\gamma_{1}a_{2} + \sum\limits_{k \in \mathbb N} \gamma_{k+1}a_{k+2}\right) + \ldots + \mu_{t}\left(t\gamma_{1}a_{t} + \sum\limits_{k \in \mathbb N} \gamma_{k+1}a_{k+t}\right)\Bigg] = \\
	= \Bigg[\lambda_{1}a_{1}, \mu_{1}\left(\sum\limits_{k \in \mathbb N} \gamma_{k}a_{k}\right) + \mu_{2}\left(2\gamma_{1}a_{2} + \sum\limits_{k \in \mathbb N} \gamma_{k+1}a_{k+2}\right) + \ldots + \\
	+ \mu_{t}\left(t\gamma_{1}a_{t} + \sum\limits_{k \in \mathbb N} \gamma_{k+1}a_{k+t}\right)\Bigg] = \\
	= \lambda_{1}\mu_{1}\gamma_{1}a_{2} + (\lambda_{1}\mu_{1}\gamma_{2} + 2\lambda_{1}\mu_{2}\gamma_{1})a_{3} + \\ 
	+(\lambda_{1}\mu_{1}\gamma_{3} + \lambda_{1}\mu_{2}\gamma_{2} + 3\lambda_{1}\mu_{3}\gamma_{1})a_{4} + \ldots + \\
	+	(\lambda_{1}\mu_{1}\gamma_{t} + \lambda_{1}\mu_{2}\gamma_{t-1} + \ldots + t\lambda_{1}\mu_{t}\gamma_{1})a_{t+1} + \\
	+ (\lambda_{1}\mu_{1}\gamma_{t+1} + \lambda_{1}\mu_{2}\gamma_{t} + \lambda_{1}\mu_{3}\gamma_{t-1} + \ldots + \lambda_{1}\mu_{t}\gamma_{2})a_{t+2}.
	\end{gathered}
	\end{equation*}
	Thus we can see that $f([x, y]) = [f(x), y] + [x, f(y)]$.
\end{proof}

\section{Proof of Theorem B}

We have $L = \bigoplus\limits_{n \in \mathbb N}Fa_{n}$, where $[a_{1}, a_{n}] = a_{n + 1}$ for all positive integer $n$, $[a_{m}, a_{k}]=$ $ = 0$ for all $m > 1$, $k \in \mathbb N$ \cite[ Corollary 2.2 and 2.1]{CKSu2017}.

If $f$ is an arbitrary derivation of $L$ then \textbf{Lemma \ref{Lemma 2.4}} shows that $f$ in basis $\{a_{n} \mid n \in \mathbb N\}$ has the following matrix
\begin{equation*}
\begingroup
\setlength\arraycolsep{6pt}
\begin{pmatrix}
\gamma_{1} & 0 & 0 & 0 &\ldots & 0 & 0 & 0 & \ldots \\
\gamma_{2} & 2\gamma_{1} & 0 & 0 &\ldots & 0 & 0 & 0 & \ldots \\
\gamma_{3} & \gamma_{2} & 3\gamma_{1} & 0 &\ldots & 0 & 0 & 0 & \ldots \\
\gamma_{4} & \gamma_{3} & \gamma_{2} & 4\gamma_{1} &\ldots & 0 & 0 & 0 & \ldots \\
\gamma_{5} & \gamma_{4} & \gamma_{3} & \gamma_{2} &\ldots & 0 & 0 & 0 & \ldots \\
\gamma_{6} & \gamma_{5} & \gamma_{4} & \gamma_{3} &\ldots & 0 & 0 & 0 & \ldots \\
\gamma_{7} & \gamma_{6} & \gamma_{5} & \gamma_{4} &\ldots & 0 & 0 & 0 & \ldots \\
\vdots& \vdots& \vdots& \vdots& \ddots& \vdots& \vdots& \vdots& \vdots \\
\gamma_{n-2} & \gamma_{n-3} & \gamma_{n-4} & \gamma_{n-5} &\ldots & (n-2)\gamma_{1} & 0 & 0 & \ldots \\
\gamma_{n-1} & \gamma_{n-2} & \gamma_{n-3} & \gamma_{n-4} &\ldots & \gamma_{2} & (n-1)\gamma_{1} & 0 & \ldots \\
\gamma_{n} & \gamma_{n-1} & \gamma_{n-2} & \gamma_{n-3} &\ldots & \gamma_{3} & \gamma_{2} & n\gamma_{1} & \ldots \\
0 & \gamma_{n} & \gamma_{n-1} & \gamma_{n-2} &\ldots & \gamma_{4} & \gamma_{3} & \gamma_{2} & \ldots \\
0 & 0 & \gamma_{n} & \gamma_{n-1} &\ldots & \gamma_{5} & \gamma_{4} & \gamma_{3} & \ldots \\
0 & 0 & 0 & \gamma_{n} &\ldots & \gamma_{6} & \gamma_{5} & \gamma_{4} & \ldots \\
0 & 0 & 0 & 0 &\ldots & \gamma_{7} & \gamma_{6} & \gamma_{5} & \ldots \\
\vdots& \vdots& \vdots& \vdots& \vdots& \vdots& \vdots& \vdots& \vdots \\
\end{pmatrix}
\endgroup
\end{equation*}

Such matrices are finitary, so we can say about the product of these matrices and hence about their commutator. Since the sum and product of two derivations of $L$ is itself a derivation, we obtain that the matrices of this kind form a subalgebra of the Lie algebra of finitary matrices. Denote this subalgebra by $\mathbf{LC}(\infty)$. Thus we obtain isomorphism $\mathbf{Der}(L)\cong \mathbf{LC}(\infty)$.

Consider the set of matrices, having the following form
\begin{equation*}
\begingroup
\setlength\arraycolsep{10pt}
\begin{pmatrix}
0 & 0 & 0 & 0 &\ldots & 0 & 0 & 0 & \ldots \\
\gamma_{2} & 0 & 0 & 0 &\ldots & 0 & 0 & 0 & \ldots \\
\gamma_{3} & \gamma_{2} & 0 & 0 &\ldots & 0 & 0 & 0 & \ldots \\
\gamma_{4} & \gamma_{3} & \gamma_{2} & 0 &\ldots & 0 & 0 & 0 & \ldots \\
\gamma_{5} & \gamma_{4} & \gamma_{3} & \gamma_{2} &\ldots & 0 & 0 & 0 & \ldots \\
\gamma_{6} & \gamma_{5} & \gamma_{4} & \gamma_{3} &\ldots & 0 & 0 & 0 & \ldots \\
\gamma_{7} & \gamma_{6} & \gamma_{5} & \gamma_{4} &\ldots & 0 & 0 & 0 & \ldots \\
\vdots& \vdots& \vdots& \vdots& \ddots& \vdots& \vdots& \vdots& \vdots \\
\gamma_{n-2} & \gamma_{n-3} & \gamma_{n-4} & \gamma_{n-5} &\ldots & 0 & 0 & 0 & \ldots \\
\gamma_{n-1} & \gamma_{n-2} & \gamma_{n-3} & \gamma_{n-4} &\ldots & \gamma_{2} & 0 & 0 & \ldots \\
\gamma_{n} & \gamma_{n-1} & \gamma_{n-2} & \gamma_{n-3} &\ldots & \gamma_{3} & \gamma_{2} & 0 & \ldots \\
0 & \gamma_{n} & \gamma_{n-1} & \gamma_{n-2} &\ldots & \gamma_{4} & \gamma_{3} & \gamma_{2} & \ldots \\
0 & 0 & \gamma_{n} & \gamma_{n-1} &\ldots & \gamma_{5} & \gamma_{4} & \gamma_{3} & \ldots \\
0 & 0 & 0 & \gamma_{n} &\ldots & \gamma_{6} & \gamma_{5} & \gamma_{4} & \ldots \\
0 & 0 & 0 & 0 &\ldots & \gamma_{7} & \gamma_{6} & \gamma_{5} & \ldots \\
\vdots& \vdots& \vdots& \vdots& \vdots& \vdots& \vdots& \vdots& \vdots \\
\end{pmatrix}
\endgroup
\end{equation*}
Denote the set of all matrices of this form by $\mathbf{NC}(\infty)$.

Further, denoteby $A$ the subset of $\mathbf{Der}(L)$, consisting of all derivations $f$ such that $f(x) \in [L, L]$ for each element $x \in L$. If $h$ is another endomorphism of $L$ such that $h(x) \in [L, L]$ for each element $x \in L$, then
\begin{equation*}
\begin{gathered}
(f-h)(x) = f(x) - h(x) \in [L, L] \\
and \\
[f,h](x) = (f \circ h - h \circ f)(x) = f(h(x)) - h(f(x)) \in [L, L]
\end{gathered}
\end{equation*}
By \textbf{Lemma \ref{Lemma 2.3}}, it follows that $A$ is a subalgebra of $\mathbf{Der}(L)$. Moreover, $A$ is an ideal of $\mathbf{Der}(L)$. Indeed, if $f \in A$ and $h$ is an arbitrary derivation of $L$, then $f(h(x)) \in [L, L]$ by the definition of $f$, and by \textbf{Lemma \ref{Lemma 2.3}}, $h(f(x)) \in [L, L]$. It follows that $\mathbf{NC}(\infty)$ is an ideal of $\mathbf{LC}(\infty)$. It is not hard to see that we can write every matrix from $\mathbf{LC}(\infty)$, in the form
\begin{equation*}
\gamma_{2}\sum\limits_{k \in \mathbb N}E_{k+1,k} + \gamma_{3}\sum\limits_{k \in \mathbb N}E_{k+2,k} + \ldots + \gamma_{n}\sum\limits_{k \in \mathbb N}E_{k+n-1,k}.
\end{equation*}
Denote by $\mathbf{DaC}(\infty)$ the subset of $\mathbf{LC}(\infty)$, having the form $\sum\limits_{k \in \mathbb N}k\gamma E_{k,k}$. It is not hard to see that $\mathbf{DaC}(\infty)$ is closed by addition and multiplication, and, moreover, the multiplication is commutative. Thus, we can consider $\mathbf{DsC}(\infty)$ as an abelian Lie
subalgebra of $\mathbf{LC}(\infty)$. The preimage of $\mathbf{DaC}(\infty)$ in $\mathbf{Der}(L)$ is subset $D$, consisting of all derivations $f$ such that $f(a_{1}) = \gamma a_{1}$. Thus we obtain that $D$ is an abelian subalgebra of $\mathbf{Der}(L)$. Clearly the mapping $\theta\colon \mathbf{DaC}(\infty) \longrightarrow F$ defined by the rule $\theta\left(\sum\limits_{k \in \mathbb N}k\gamma E_{k,k}\right) = \gamma$, is an isomorphism. It shows that $\mathbf{DaC}(\infty)$ is an abelian subalgebra of a Lie algebra
$\mathbf{LC}(\infty)$ and $\mathbf{DaC}(\infty)$ is isomorphic to a field $F$.

It is clear that every matrix from $\mathbf{LC}(\infty)$ is a sum of a matrix from $\mathbf{NC}(\infty)$ and a matrix from $\mathbf{DaC}(\infty)$. This means that a Lie algebra $\mathbf{LC}(\infty)$ is a sum of ideal $\mathbf{NC}(\infty)$ and abelian subalgebra $\mathbf{DaC}(\infty)$, and, moreover, their intersection is zero.

Let $\Gamma, \Lambda \in \mathbf{NC}(\infty)$, where 
\begin{equation*}
\begin{gathered}
\Gamma = \sum\limits_{t \in \mathbb N} \left(\gamma_{t}\sum\limits_{k \in \mathbb N}E_{k+t,k}\right), \\
\Lambda = E +\sum\limits_{t \in \mathbb N} \left(\lambda_{t}\sum\limits_{k \in \mathbb N}E_{k+t,k}\right)
\end{gathered}
\end{equation*}
(here only finitely many coefficients $\gamma_{t}, \lambda_{t}$ are non-zero). As in \textbf{Theorem \ref{Theorem A}}, it is possible to prove that $\Gamma\Lambda \in \mathbf{NC}(\infty)$ and $\Gamma\Lambda = \Lambda\Gamma$. Hence, ideal $\mathbf{NC}(\infty)$ is abelian. Isomorphism $\mathbf{NC}(\infty)\cong A$ show that $A$ is also abelian.

Finally, let $M = \lVert\alpha_{kt}\rVert_{k,t \in \mathbb N}$ where only finitely many coefficients $\alpha_{kt}$, $k\in~\mathbb N$ are non-zero for every index $t \in \mathbb N$, and $D = \lVert\delta_{kt}\rVert_{k,t \in \mathbb N}$, where $\delta_{kt} = 0$ \linebreak[5] whenever $k \ne t$. Put
\begin{equation*}
\begin{gathered}
\Lambda = DM = \lVert\lambda_{kt}\rVert_{k, t \in \mathbb N},\\ 
P = MD = \lVert\rho_{kt}\rVert_{k, t \in \mathbb N}, \\ 
\Sigma = [D, M] = DM - MD= \lVert\sigma_{kt}\rVert_{k,t \in \mathbb N}.
\end{gathered}
\end{equation*}
We have
\begin{equation*}
\begin{gathered}
\lambda_{kt} = \sum\limits_{n \in \mathbb N}\delta_{kn}\alpha_{nt} = \delta_{kk}\alpha_{kt}, \\
\rho_{kt} = \sum\limits_{n \in \mathbb N}\alpha_{kn}\delta_{nt} = \alpha_{kt}\delta_{tt}, \\
\sigma_{kt} = \lambda_{kt} - \rho_{kt} = \delta_{kk}\alpha_{kt} - \alpha_{kt}\delta_{tt}= \alpha_{kt}(\delta_{kk} - \delta_{tt}).
\end{gathered}
\end{equation*}
In particular, if
\begin{equation*}
\begin{gathered}
\Gamma = \gamma_{1}\sum\limits_{k \in \mathbb N}E_{k+1,k} + \gamma_{2}\sum\limits_{k \in \mathbb N}E_{k+2,k} + \ldots + \gamma_{n}\sum\limits_{k \in \mathbb N}E_{k+n,k},\\
D = \sum\limits_{k \in \mathbb N}k\mu E_{k,k},
\end{gathered}
\end{equation*}
then
\begin{equation*}
\begin{gathered}
\sigma_{kt} = \alpha_{kt}(\delta_{kk} - \delta_{tt}) = \alpha_{kt}(k\mu - t\mu) =  \mu\alpha_{kt}(k - t).
\end{gathered}
\end{equation*}
Find the first column of the matrix $[D,\Gamma]$. As we saw earlier it defines this matrix. We have
\begin{equation*}
\begin{gathered}
\sigma_{21} = \mu\gamma_{1}, \sigma_{31} = 2\mu\gamma_{2}, \sigma_{41} = 3\mu\gamma_{2},\ldots , \sigma_{n + 1,1} = n\mu\gamma_{n}.
\end{gathered}
\end{equation*}

Suppose that $\mu\ne 0$ (if $\mu = 0$, then $D = 0$). If $\mathbf{char}(F) = 0$, then put
\begin{equation*}
\begin{gathered}
\Theta = \theta_{1}\sum\limits_{k \in \mathbb N}E_{k+1,k} + \theta_{2}\sum\limits_{k \in \mathbb N}E_{k+2,k} + \ldots + \theta_{n}\sum\limits_{k \in \mathbb N}E_{k+n,k}, \\ where \\
\theta_{1} = \mu^{-1}\gamma_{1}, \theta_{2} = \frac{1}{2} \mu^{-1}\gamma_{2}, \ldots ,\theta_{n} = \frac{1}{n}\mu^{-1}\gamma_{n}.
\end{gathered}
\end{equation*}
Then by above proved $[D, \Theta] = \Gamma$. It follows that $[D, \mathbf{NC}(\infty)] = \mathbf{NC}(\infty)$.


\end{document}